\documentclass[11pt, reqno]{amsart}
\usepackage{fullpage, amsfonts,amsmath,amscd,amssymb}

\usepackage[enableskew]{youngtab}
\input xy
\xyoption{all}

\newtheorem*{lemma}{Lemma}
\newtheorem*{proposition}{Proposition}
\newtheorem*{theorem}{Theorem}
\newtheorem*{corollary}{Corollary}

\newtheorem*{ypoth}{Hypothesis}

\theoremstyle{remark}

\newtheorem*{remark}{Remark}

\newcommand{\irr}[1]{\textsf{Irrep}(#1)}

\newcommand{\ff}{\footnote}

\DeclareMathOperator{\Hom}{Hom}

 \DeclareMathOperator{\id}{Id}
 
 \DeclareMathOperator{\md}{mod}

\newcommand{\Z}{\mathbb{Z}}
\newcommand{\C}{\mathbb{C}}

\DeclareMathOperator{\hr}{\mathfrak{h}^{\text{reg}}}

\usepackage[mathscr,mathcal]{eucal}
\usepackage{colortbl}
\usepackage{amsfonts}
\usepackage{amssymb}
\usepackage{enumerate}

\title{Cell modules and canonical basic sets for Hecke algebras from Cherednik algebras}
\dedicatory{To Ken Goodearl, on his 65th birthday}
\author{Maria Chlouveraki}
\address{ School of Mathematics and Maxwell Institute of Mathematics, University of Edinburgh, Edinburgh, EH9 3JZ (M.C. \& I.G.); Instituto de Matem\'atica y F\'isica, Universidad de Talca, 
Talca, Chile (S.G)} \email{maria.chlouveraki@ed.ac.uk, igordon@ed.ac.uk, sgriffeth@ed.ac.uk}
\author{Iain Gordon}
\author{Stephen Griffeth}
 
\thanks{We would like to thank C\'edric Bonnaf\'e, Nicolas Jacon and Rapha\"el Rouquier for helpful conversations. The first author is grateful for the financial support of EPSRC grant EP/G04984X and the second and third authors are grateful for the financial support of EPSRC grant EP/G007632; the second author warmly acknowledges the hospitality of the Hausdorff Institute for Mathematics.}
\begin{document}
\maketitle
\section{Introduction}
\subsection{} One parameter Iwahori-Hecke algebras of finite Coxeter groups have Kazhdan-Lusztig bases of geometric origin; the same is predicted for unequal parameters, and wished for complex reflection groups. These bases are fundamental in Lie theory and play a significant role in the representation theory of Iwahori-Hecke algebras. 

\subsection{} In this note we are interested in labelling the irreducible representations of non-semisimple specialisations of Hecke algebras of complex reflection groups. We will use category $\mathcal{O}$ for the rational Cherednik algebra and the ${\sf KZ}$ functor together with elementary algebraic and combinatorial arguments to construct {\it canonical basic sets} in many cases -- see \S \ref{Basic Sets} for the definition. For finite Coxeter groups our observations can be stated as follows. (Similar statements hold for complex reflection groups of type $G(\ell ,1,n)$.)

\begin{theorem}
Let $W$ be a finite Coxeter group and $\mathcal{H}(W)$ be the corresponding Hecke algebra with unequal parameters specialised in $\C$. Let ${\sf KZ} : \mathcal{O}(W) \to \mathcal{H}(W)$ denote the ${\sf KZ}$ functor, $L(E)$ and $\Delta(E)$ the irreducible and standard representations in $\mathcal{O}(W)$. Then $\{ {\sf KZ} (L(E)): E\in \irr{W}, {\sf KZ}(L(E))\neq 0 \}$ is a canonical basic set for $\mathcal{H}(W)$ and there is a symmetric bilinear form on ${\sf KZ}(\Delta(E))$ which is zero or factors through ${\sf KZ}(L(E))$.
\end{theorem} 

\subsection{} Regarding the first claim above, the existence of canonical basic sets for finite Coxeter groups has been known for some time thanks to work of Jacon and others, \cite{GeJa}, \cite{jaca}. Existing proofs in type $B$, however, use Ariki's theorem on the categorification of Hecke algebra representations; our proof does not. If we use the earlier work we get an explicit combinatorial description of the irreducible representations in $\mathcal{O}(W)$ killed by ${\sf KZ}$; this appears to be new. 

\subsection{} The existence of symmetric bilinear forms on modules is also not surprising. Geck has shown that all Iwahori-Hecke algebras of finite Coxeter groups admit at least one cellular algebra structure, \cite{geck}, \cite{Geck2}. In the unequal parameter case, it is even expected that there are several different cellular structures depending on the choice of a weight function defining the Hecke algebra; Geck has proved this provided a series of conjectures of Lusztig hold. As a consequence each Hecke algebra is expected to admit a family of cell modules, depending on the choice of weight function, and these will carry a symmetric bilinear form such that the quotient by the radical of this form is either zero or irreducible. But this is precisely what ${\sf KZ}$ of the standard modules of rational Cherednik algebras do, without assumptions. We show that these modules agree with the appropriate cell modules, whenever the cell modules exist. It is worth pointing out that
  Lusztig's conjectures are not yet known to hold for type $B$ in general. They are, however, known to hold for ``dominant" choices of weight function and in this case \cite{GeIaPa} shows the cell modules are the Specht modules of \cite{DJM}. So, for a special choice of parameters in this case -- and more generally for $G(\ell , 1, n)$ -- we can identify the image of the standard modules under ${\sf KZ}$ with these Specht modules; in general they are different.

\subsection{} The paper is organised as follows. In the following section we recall the definition of Hecke algebras for complex reflection groups and category $\mathcal{O}$ for rational Cherednik algebras. In Section 3 we discuss basic sets, while in Section 4 we prove that ${\sf KZ}$ of the standard modules have symmetric bilinear forms and compare this with existing results in the finite Coxeter group case. We end by studying the $G(\ell ,1 ,n)$ case where we require combinatorial arguments to produce complete results.
\section{Hecke and Cherednik algebras}
\subsection{Notation} \label{notation}Let $W$ be a complex reflection group with reflection representation $\mathfrak{h}$. Let $\mathcal{A}$ be the set of reflecting hyperplanes in $\mathfrak{h}$. 

For $H\in \mathcal{A}$ let $W_H$ be the pointwise stabiliser of $H$ in $W$, set $e_H = |W_H|$ and 
let $U = \bigcup_{H\in \mathcal{A}/W}\irr{W_H}$. Since $W_H$ is a finite cyclic group, we may identify elements of $U$ with pairs $(H,j)$ where $0\leq j < e_H$ and the irreducible representation of $W_H$ is given by $\det^j|_{W_H}$.

Given $H\in \mathcal{A}$, choose $\alpha_H\in \mathfrak{h}^*$ with $\ker \alpha_H = H$ and let $v_H\in \mathfrak{h}$ be such that $\C v_H$ is a $W_H$-stable complement to $H$. Let $\hr = \mathfrak{h}\setminus \bigcup_{H\in \mathcal{A}} H$ and $B_W = \pi_1(\hr/W, x_0)$ where $x_0$ is some fixed basepoint.

\subsection{} For any positive integer $e$ we will write $\zeta_e$ for $\exp(2\pi \sqrt{-1}/e)\in \C$.

\subsection{Hecke algebras} Let $\{ {\bf q}_u\}$ be a set of indeterminates with $u\in U$ and set ${\bf k} = \mathbb{C}[\{ {\bf q}_{u}^{\pm 1}\}]$. 

Let $\mathcal{H}$ be the Hecke algebra of $W$ over ${\bf k}$, the quotient of ${\bf k}[B_W]$ by the relations $$\prod_{0\leq j < e_H} (T_H - \zeta_{e_H}^j{\bf q}_{H,j}) = 0, $$ where there is a relation for each $H\in \mathcal{A}$ and where $T_H$ is a  generator for the monodromy around $H$, see for instance \cite[\S 4]{BMR}. 

\begin{ypoth}
The algebra $\mathcal{H}$ is free over ${\bf k}$, of rank $|W|$. There is a symmetrising form $t: \mathcal{H} \longrightarrow {\bf k}$ that becomes the canonical symmetrising form on $ \C[W]=  \mathcal{H}\otimes_{\bf k} \C $ on specialising ${\bf q}_{H,j}$ to $1$. 
\end{ypoth}

This hypothesis is known to hold for all but finitely many irreducible complex reflection groups, and it is conjectured to be true in general \cite[\S 4C]{BMR}.

\subsection{} \label{a defn} Given any $\C$-algebra homomorphism $\Theta: {\bf k} \rightarrow k$, we will let $\mathcal{H}_{\Theta}$ denote the specialised algebra $\mathcal{H}\otimes_{\bf k} k$. 
The cyclotomic specialisation $\Theta : {\bf k} \rightarrow \C[q^{\pm 1}]$ will be important to us. For this we pick a set of integers $m=\{m_u\}$ and then send ${\bf q}_{u}$ to $q^{m_u}$, where $q$ is either an indeterminate or a non-zero complex number. We will denote this  $\mathcal{H}_{\Theta}$ by $\mathcal{H}_{q,m}$. 

\subsection{} There is a positive integer $n$ so that after adjoining an $n$th root $z$ of $q$, $\mathcal{H}_{q,m}$ becomes split semisimple.  Let $E \in \irr{W}$ and let $s_E \in \C[z^{\pm 1}]$ be the associated Schur element of $\mathcal{H}_{q,m}$, where $q$ is an indeterminate (see \cite[\S 2B]{BMM}). We set
\begin{center}
$a_E= - \mathrm{val}_q(s_E)=- \mathrm{val}_z(s_E)/n$ \,\,\,and\,\,\, $A_E= - \mathrm{deg}_q(s_E)=- \mathrm{deg}_z(s_E)/n.$
\end{center}

\subsection{Cherednik algebras}

Let $\{{\bf h}_u\}$ be a set of indeterminates with $u\in U$ and set ${\bf R} = \C [\{{\bf h}_u\}]$. Let ${\bf H}$ be the rational Cherednik algebra over ${\bf R}$ attached to $W$, see \cite[\S 5]{R} whose notation we follow. As an ${\bf R}$-algebra ${\bf H}$ has a triangular decomposition ${\bf R}[\mathfrak{h}]\otimes_{\bf R}  {\bf R}[W]\otimes_{\bf R} {\bf R}[\mathfrak{h}^*]$. The commutation relation between $y\in \mathfrak{h} \subset {\bf R}[\mathfrak{h^*}]$ and $x\in \mathfrak{h}^* \subset {\bf R}[\mathfrak{h}]$ is given by $$[y,x] = \langle y, x\rangle + \sum_{H\in\mathcal{A}} \frac{\langle y, \alpha_H\rangle\langle v_H, x\rangle}{\langle v_H,\alpha_H\rangle}\gamma_H$$ where $$\gamma_H = \sum_{w\in W_H\setminus\{1\}} \left(\sum_{j=0}^{e_H-1} \det(w)^{-j} ({\bf h}_{H,j} - {\bf h}_{H,j-1})\right) w.$$

Again, given any $\Psi: {\bf R}\rightarrow R$, we define ${\bf H}_{\Psi} = {\bf H}\otimes_{\bf R} R$.

\subsection{Category $\mathcal{O}$ and the ${\sf KZ}$ functor} Let $\Psi : {\bf R}\rightarrow R$ with $R$ a local commutative noetherian algebra with residue field $K$, and let $\psi: {\bf R} \rightarrow K$ be the extension of $\Psi$ to $K$. 
Given $E\in \irr{W}$, set ${c}_E\in K$ to be the scalar by which the element $-\sum_{H \in \mathcal{A}} \sum_{j=0}^{e_{H}-1} \left( \sum_{w \in W_H} (\det w)^{-j}w \right) \psi({\bf h}_{H,j})\in Z(K[W])$ acts on $E\otimes_{\C} K$\ff{In the rational Cherednik algebra literature, including \cite{R}, the function $c_E$ is usually taken to be the negative of the $c_E$ here; but in the context of this paper the above definition is more natural.}.

\subsection{} Set $\mathcal{O}_{\Psi}$ to be the category of finitely generated ${\bf H}_{\Psi}$-representations that are locally nilpotent for the action of $\mathfrak{h}\subset R[\mathfrak{h}^*]$. This is a highest weight category, \cite{GGOR} and \cite[\S 5.1]{R}. Its standard objects are $\Delta_{\Psi}(E) = {\bf H}_{\Psi}\otimes_{R[\mathfrak{h}^*]\rtimes W} (R\otimes_{\C} E)$ where $E\in \irr{W}$, and its ordering is defined by $$\Delta_{\Psi}(E) < \Delta_{\Psi}(F) \text{ if and only if } {c}_F - {c}_E \in \mathbb{Z}_{>0}.$$
Henceforth we will write $E<_{\Psi} F$ if $c_F - c_E \in \mathbb{Z}_{>0}$.

\subsection{} \label{KZappears} Let $\hat{R}$ be the completion of $\C[\{{\bf h}_u\}]$ at a maximal ideal corresponding to the point $\{ h_u \}\in \C^U$. Consider $\hat{R}$ as a ${\bf k}$-algebra via the homomorphism that sends ${\bf q}_{H,j}$ to $\exp(2\pi \sqrt{-1} {\bf h}_{H,j})$. Thus for any homomorphism $\Psi : {\bf R} \rightarrow R$ which factors through $\hat{R}$ there is a corresponding homomorphism $\Theta : {\bf k}\rightarrow R$. Then there is an exact functor $${\sf KZ}_{\Psi}: \mathcal{O}_{\Psi} \rightarrow \mathcal{H}_{\Theta}-\md$$

\section{Category $\mathcal{O}$ and basic sets}

\subsection{} Let $\{h_u\}\in \C^U$ and let $q_u= \exp(2\pi \sqrt{-1} h_u)$ for each $u\in U$. Let $\psi: {\bf R}\rightarrow \C$ and $\theta: {\bf k}\rightarrow \C$ be the corresponding specialisation maps. In this case each $\Delta_{\psi}(E)$ has an irreducible head which we write as $L_{\psi}(E)$. We define two sets of $\mathcal{H}_\theta$-representations 
\begin{equation} \label{cell and simple} S_q(E)={\sf KZ}_{\psi}(\Delta_{\psi}(E)) \textrm{ and } D_q(E)={\sf KZ}_{\psi}(L_{\psi}(E)).\end{equation}
Set ${\bf B} \subseteq \irr{W}$ to be the $E\in\irr{W}$ such that $D_{q}(E) \neq 0$.
For $E \in {\bf B}$, set $$\tilde{c}_E:=\mathrm{min}\{c_F \,:\,F \in \irr{W} \textrm{ such that } [S_q(F):D_q(E)] \neq 0\}.$$

\begin{proposition}\label{can bas set} 
 \begin{enumerate}[(a)] \item The set
$\{D_q(E):E \in {\bf B} \}$ is a complete set of pairwise non-isomorphic irreducible $\mathcal{H}_{\theta}$-representations.
\item If $E\in {\bf B}$, then
$[S_q(E):D_q(E)] =1$.
\item For $E\in {\bf B}$, we have $c_{E}=\tilde{c}_E$;
\item  If $[S_q(F):D_q(E)] \neq 0$ for some $F \in \mathrm{Irr}(W)$ and $E\in {\bf B}$, then either 
$F=E$ or $\tilde{c}_E<c_F$.
\end{enumerate}

\end{proposition}

\begin{proof} Part (a) follows from \cite[Theorem 5.14]{GGOR}. Then since the functor ${\sf KZ}_{\Psi}$ is exact, we have
$$
[S_q(F): D_q(E)]
=
[\Delta_{\psi}(F):L_{\psi}(E)] \quad \textrm{if} \quad {\sf KZ}_\psi(L_{\psi}(E)) \neq 0.
$$
Since $\mathcal{O}_{\psi}$ is a highest weight category, the composition series of $\Delta_{\psi}(F)$ consists of $L_{\psi}(E)$'s with $E \leq_{\psi} F$ and we have $[\Delta_{\psi}(E):L_{\psi}(E)]=1$. So (b) holds. 
Now, let $E\in {\bf B}$ and 
$
\mathcal{S}_E=\{F \in \irr{W} \,:\,[S_q(F):D_q(E)] \neq 0\}
=
\{F \in \irr{W} \,:\, [\Delta_{\psi}(F):L_{\psi}(E)] \neq 0\}.
$
We have $E \in \mathcal{S}_E$. If $F \in \mathcal{S}_E$ with $F\neq E$, then 
$E <_\psi F$, whence $c_E < c_F$. Therefore, $c_E=\tilde{c}_E$ and (c) and (d) hold.
\end{proof}
\subsection{Basic Sets} \label{Basic Sets}In the situation of the above proposition, we say that ${\bf B}$ is a \emph{basic set with respect to} $c$. More generally, if we have another function $f: \irr{W} \longrightarrow \C$ and a subset ${\bf B}'\subseteq \irr{W}$ that satisfy the properties of Proposition \ref{can bas set} with $f$ replacing $c$, then we say that ${\bf B}$ is a \emph{basic set with respect to} $f$. A basic set with respect to $f$ is unique. 

In the case that we let $q_u = \exp(2\pi\sqrt{-1} m_u)$ for integers $\{ m_u \}\in \Z^U$ and we use the $a$-function defined in \ref{a defn}, then, using the ${\sf KZ}$-functor to identify $\irr{W}$ with $\irr{\mathcal{H}_{\text{Quot}({\bf k})}}$, a basic set with respect to $a$ is exactly the canonical basic set in the sense of Geck-Rouquier \cite{GeRo}.

\subsection{The functions $a$, $A$ and $c$} Recall from \ref{notation} the basepoint  $x_0 \in \mathfrak{h}^{\mathrm{reg}}$, and from \ref{a defn} the cyclotomic Hecke algebra $\mathcal{H}_{q,m}$ over $\C[q^{\pm 1}]$ where $w$ is an indeterminate. Let $K = \C(q)$, so that $\mathcal{H}_{K,m} = K\otimes_{\C[q^{\pm 1}]} \mathcal{H}_{q,m}$ is split semisimple, with a bijection between $\irr{W}$ and $\irr{\mathcal{H}_{K,m}}$. Let ${\boldsymbol{\pi}}$ be the central element of the pure braid group $P_W = \pi_1(\hr, x_0)$ defined by the loop $s \mapsto \mathrm{exp}(2\pi \sqrt{-1} s)x_0$ and let $\omega_{E}(\boldsymbol{\pi})$ denote the scalar in $\C[q^{\pm 1}]$ by which $\boldsymbol{\pi}$ acts on the irreducible representations of $\mathcal{H}_{K,m}$ corresponding to $E$.  Recall $t$ denotes the canonical symmetrizing form on $\mathcal{H}_{q,m}$. By combining \cite[Proposition 2.8]{BMM} and the formulas for $\omega_{E}(\boldsymbol{\pi})$ and $t(\boldsymbol{\pi})$ in \cite[(1.22) and Theorem 2.1.2(b) respectively]{BMM} with \cite[(6.5) and Lemma 6.6]{BMM}, we have
$$\frac{\omega_{E}(\boldsymbol{\pi})}{t(\boldsymbol{\pi})}=q^{-a_E-A_E}.$$ 

Setting $h_{H,j} = m_{H,j}$ gives a specialisation map $\psi : {\bf R} \rightarrow \C$.  The formulas for $\omega_{E}(\boldsymbol{\pi})$ and $t(\boldsymbol{\pi})$ mentioned above then show that  \begin{equation} \label{a+A}
 a_E+A_E=\frac{d}{dq}\left(\frac{\omega_{E}(\pi)}{t(\pi)}\right) (1)=c_E + \sum_{H\in \mathcal{A}}\sum_{j=0}^{e_H-1} m_{H,j}.
 \end{equation}
Thus we see that orderings on $\irr{W}$ determined by the functions $c$ or $a+A$ are equal.

\section{Forms and Standard Modules}

\subsection{}Let $\{ h_{u}\} \in \mathbb{R}^U$, $\hat{R}$ the completion of ${\bf R}$ at $\{h_u\}$ and $q_u = \exp(2\pi\sqrt{-1}h_u)$ for all $u\in U$. We have homomorphisms $$ \xymatrix{ {\bf R} \ar[r]^{\Psi} \ar[dr]_{\psi} & \hat{R} \ar[d]& \text{and}  & {\bf k} \ar[r]^{\Theta} \ar[dr]_{\theta} & \hat{R} \ar[d] \\ &  \C & &  & \C.}$$ Let $Q = \text{Quot}(\hat{R})$ and $\Theta_Q: {\bf k} \rightarrow \hat{R} \rightarrow Q$ the corresponding embedding. Setting $E^{\dagger}=Q\otimes_{\hat{R}} {\sf KZ}_{\Psi}(\Delta_\Psi(E))$ gives a bijection $E \mapsto E^\dagger$ from $\irr{W}$ to $\irr{{\mathcal H}_{\Theta_Q}}$.

\subsection{} \label{assumfornow}  We will assume for the next few sections that there is a ${\bf k}$-algebra anti-involution $\sigma$ of $\mathcal{H}$.  It passes to any specialisation of $\mathcal{H}$.  We assume further that there is a non-degenerate symmetric bilinear form $(\, \,,\,): E^{\dagger}\otimes E^{\dagger}\rightarrow Q$ that satisfies $(he,e') = (e, \sigma(h)e')$ for all $e,e'\in E^{\dagger}$ and $h\in\mathcal{H}_{\theta_Q}$. 

\subsection{} We do not know how restrictive the assumptions of  \ref{assumfornow}  are. However, it is elementary that they hold for all Coxeter groups and for the groups of type $G(\ell , 1, n)$:
\begin{itemize}
\item
For $W$ a finite Coxeter group we take $\sigma (T_w) = T_{w^{-1}}$. The non-degenerate symmetric form arises, for instance, as in the first part of the proof of \cite[Proposition 4.3]{geckconstrchar}.
\item
For $W$ be a complex reflection group of type $G(\ell ,1,n)$ we take $\sigma(T_i) = T_i$ for $0\leq i \leq n-1$, where the $T_i$ are the usual generators, defined for instance in \cite[Theorem 3.6]{BMR}. This induces an anti-involution of $\mathcal{H}$ and the form may be deduced from \cite[p.398]{DJM}.\end{itemize}

\subsection{} Recall the definition of the $\mathcal{H}_{\psi}$-representations $S_q(E)$ and $D_q(E)$ given in \eqref{cell and simple}. The following theorem is motivated by \cite[Proposition 6.5]{GGOR}. 

\begin{theorem} \label{forms} Assume that $W$ satisfies the assumption of \ref{assumfornow}. Then there is a symmetric bilinear form $\langle \,\, , \, \rangle: S_q(E)\otimes S_q({E})\rightarrow \C$ which, for all $s,s' \in S_q(E)$ and $h\in \mathcal{H}_{\theta}$, satisfies $\langle h s, s' \rangle = \langle s, \sigma(h) s'\rangle$, and such that $D_q(E) \cong S_q(E)/ \text{rad}  \langle \,\, , \, \rangle$. 
\end{theorem}
\begin{proof}
The functor ${\sf KZ}_{\Psi}$ is represented by a projective object $P_{{\sf KZ},\Psi}\in \mathcal{O}_{\Psi}$, \cite[\S 5.4]{GGOR} and \cite[\S 5.2]{R}. The highest weight structure on $\mathcal{O}_{\Psi}$ gives a finite filtration of $P_{{\sf KZ},\Psi}$  $$0\subseteq \cdots \subseteq P^{i+1} \subseteq P^{i} \subseteq P^{i-1} \subseteq \cdots \subseteq P_{{\sf KZ},\Psi},$$ where each section $P^{i}/P^{i+1}$ is a direct sum of standard modules $\Delta_\Psi(E)$ for which $i\leq {c}_E < i+1$, each appearing with multiplicity $\dim (E)$, \cite[Proof of Theorem 5.15]{GGOR}. We apply ${\sf KZ}_{\Psi}$ to this to get a filtration of ${\sf KZ}_\Psi(P_{{\sf KZ}, {\Psi}}) ={\mathcal H}_{\Theta}$ whose sections are direct sums of ${\sf KZ}_{\Psi}(\Delta_\Psi(E))$ for various $E$. 

The above filtration of ${\mathcal H}_\Theta$ then induces a filtration of ${\mathcal H}_{\Theta_Q}$ whose sections are direct sums of $(E^{\dagger})^{\dim E}$ where $i\leq {c}_E < i+1$. Define another filtration on ${\mathcal H}_{\Theta_Q}$ by setting $(Q\otimes_{\bf R} {\mathcal H})^{\geq i}$ to be the sum of the $E^{\dagger}$-isotypic components of ${\mathcal H}_{\Theta_Q}$ where $c_E\geq i$. This induces a filtration ${\mathcal H}_{\Theta}^{\geq i} = {\mathcal H}_{\Theta} \cap (Q\otimes_{\bf R} {\mathcal H})^{\geq i}$ of ${\mathcal H}_{\Theta}$. 

We claim this equals the above filtration by ${\sf KZ}_{\Psi}(P^i)$'s. By construction $Q\otimes_{\hat R} {\sf KZ}_{\Psi}(P^i) = (Q\otimes_{\bf R} {\mathcal H})^{\geq i}$ so we deduce for each $i$ that ${\sf KZ}_{\Psi}(P^i)\subseteq {\mathcal H}_{\Theta}^{\geq i}$ and the quotient is torsion. But ${\mathcal H}_{\Theta}^{\geq i}/{\sf KZ}_{\Psi}(P^i) \subseteq {\mathcal H}_{\Theta}/{\sf KZ}_{\Psi}(P^i)$ and this is a free $\hat{R}$-module since it has a filtration by ${\sf KZ}_{\Psi}(\Delta_{\Psi}(E))$'s, each of which is free over $\hat{R}$. Thus there is no torsion, and we have the claimed equality.

By \cite[Proposition 5.21]{GGOR} and \cite[Proposition 4.30]{R} there is a $\nabla$-filtration of $P_{{\sf KZ}, \Psi}$  $$0\subseteq \cdots \subseteq P_{j-1}\subseteq P_j \subseteq P_{j+1}\subseteq \cdots \subseteq P_{{\sf KZ}, \Psi},$$ where each section $P_{j+1}/P_{j}$ is a direct sum of $\nabla_{\Psi}(E)$'s with $j\leq {c}_{E} < j+1$. Let $K_j = \Hom_{\hat{R}}({\sf KZ}_{\Psi}(P_{{\sf KZ}, \Psi}/P_j), \hat{R})$ so that we have a filtration of right ${\mathcal H}_{\Theta}$-modules $$0 \subseteq \cdots \subseteq K_{j+1}\subseteq K_j \subseteq K_{j-1} \subseteq \cdots \subseteq \Hom_{\hat R}({\sf KZ}_{\Psi}(P_{{\sf KZ},\Psi}),\hat{R})= \Hom_{\hat R}({\mathcal H}_{\Theta},{\hat R}).$$ We consider all of the above right ${\mathcal H}_{\Theta}$-modules as left modules via $\sigma.$ As the mapping $x\mapsto \left(y \mapsto t(\sigma(x)y)\right)$ from ${\mathcal H}_{\Theta}$ to $\Hom_{\hat R}({\mathcal H}_{\Theta},\hat{R})$ is a left ${\mathcal H}_{\Theta}$-isomorphism, we then have a filtration $K_{\bullet}$ of $\mathcal{H}_{\Theta}$ by left $\mathcal{H}_{\Theta}$-modules. 

We claim that each $K_j = \mathcal{H}_{\Theta}^{\geq j}$. All $K_j$ are free over $\hat{R}$ since each $\nabla(E)$ is a free $\hat{R}$-module, \cite[Proposition 4.19]{R} and ${\sf KZ}_{\Psi}(-) = \Hom_{\mathcal{O}_{\Psi}}(P_{{\sf KZ},\Psi}, -)$ and so preserves $\hat{R}$-projectivity. Thus the sections $K_{j}/K_{j+1} = \Hom_{\hat{R}}({\sf KZ}_{\Psi}(P_{j+1}/P_{j}), \hat{R})$ are a direct sum of $\Hom_{\hat R}({\sf KZ}_{\Psi}(\nabla_\Psi(E)),\hat{R})$'s with $j \leq c_{E} < j+1$. On extending scalars to $Q$ the filtration thus has sections that are direct sums of $(E^{\dagger\ast})^{\dim E} = ({E}^{\dagger})^{\dim {E}}$ where $j\leq {c}_{E} < j+1$. Furthermore, the quotients $\mathcal{H}_{\Theta}/K_j$ are $\hat{R}$-free. Hence arguing as in the third paragraph of this proof we see that $K_j = {\mathcal H}_{\Theta}^{\geq j}$ as claimed.

It follows that $K_j = P^j$ and so $\Hom_{\hat R}({\sf KZ}_{\Psi}(\nabla_\Psi(E)), \hat{R}) \cong {\sf KZ}_{\Psi}(\Delta_\Psi(E'))$ for some $E'\in \irr{W}$. Passing again to $Q$ we see that $E'$ must have the property that $E'^{\dagger} \cong E^{\dagger\ast}$.  By our assumption \ref{assumfornow} we have $E^{\dagger \ast} \cong E^\dagger$ so that $E' = E$ and $\Hom_{\hat R}({\sf KZ}_{\Psi}(\nabla_\Psi({E})), \hat{R}) \cong {\sf KZ}_{\Psi}(\Delta_\Psi({E})).$ 

We now specialise the above isomorphism to an isomorphism of ${\mathcal H}_{\theta}$-representations $$\Hom_\C({\sf KZ}_{\psi}(\nabla_{\psi}(E)), \C) \cong {\sf KZ}_{\psi}(\Delta_{\psi}({E})).$$ In $\mathcal{O}_{\psi}$ there is a unique (up to a scalar) non-zero homomorphism $\Delta_{\psi}(E) \rightarrow \nabla_{\psi}(E)$, \cite[Proposition 4.19]{R}, and this factors through $L_{\psi}(E)$. Applying ${\sf KZ}_{\psi}$ to this we find a homomorphism \begin{equation} \label{form} \phi_E: {\sf KZ}_{\psi}(\Delta_{\psi}(E)) \rightarrow {\sf KZ}_{\psi}(\nabla_{\psi}(E)) \cong {\sf KZ}_{\psi}(\Delta_{\psi}({E}))^*\end{equation} which factors through ${\sf KZ}_{\psi}(L_{\psi}(E))$. 

By definition $S_q(E) = {\sf KZ}_{\psi}(\Delta_{\psi}(E))$. So \eqref{form} induces a bilinear form $S_q(E)\otimes S_q({E}) \rightarrow \C$ via $\langle s, s'\rangle = \phi_E (s)(s')$. By construction it satisfies $\langle hs, s'\rangle = \langle s, \sigma(h)s'\rangle$ for all $h\in {\mathcal H}_{\theta}$. This proves the first part of the theorem.

We now prove that this form is symmetric. By construction, the homomorphism $\phi_E$ arises from applying ${\sf KZ}_{\Psi}$ to the canonical homomorphism $\Delta_{\Psi}(E) \rightarrow \nabla_{\Psi}(E)$. Denote this by $\Phi_E: {\sf KZ}_{\Psi}(\Delta_{\Psi}(E)) \rightarrow {\sf KZ}_{\Psi}(\nabla_{\Psi}(E)) \cong {\sf KZ}_{\Psi}(\Delta_{\Psi}({E}))^*$. We then have that $\langle s, s' \rangle$ equals the image of $\Phi_E(\hat{s})(\hat{s'})$ in the residue field $\C$ where $\hat{s}$ and $\hat{s}'$ are lifts of $s,s'$ to ${\sf KZ}_{\Psi}(\Delta_{\Psi}(E))$. We can extend $\Phi_E$ to an $\mathcal{H}_{\Theta_Q}$-mapping $\Phi_E\otimes_{\hat R} \id_{Q}: E^\dagger \cong {\sf KZ}_\Psi(\Delta_{\Psi}({E}))\otimes_{\hat{R}} Q \rightarrow {\sf KZ}_\Psi(\nabla_{\Psi}({E}))\otimes_{\hat R} Q \cong E^{\dagger\ast}$. By assumption this must be a scalar multiple of the isomorphism induced by the form $(\,\,,\,)$ on $E^{\dagger}$. As this form is symmetric, it follows that $\Phi_E(\hat{s})(\hat{s'}) = \Phi_E(\hat{s}')(\hat{s})$ and this implies the symmetry of $\langle \,\,,\,\rangle$. 

Finally the radical of the form $\langle \,\,,\,\rangle$ is the kernel of the homomorphism $\Phi_E$. By construction the quotient of $S_q(E)$ by this is  ${\sf KZ}_{\psi}(L_{\psi}(E)) = D_q(E)$. 
\end{proof}

\begin{remark} We have assumed that $\{ h_u \} \in \mathbb{R}^U$ so that the value of $c_E$ is real for all $E\in \irr{W}$. The same proof works, however, if we assume that $\{ h_u \} \in z\mathbb{R}^U$ for some complex number $z$, and it is this more general form that will be used in Proposition \ref{geckcell} (if $\lambda$ there is not real). It is more painful to write down a version of this theorem when $\{ h_u\}$ is an arbitrary complex $U$-tuple.
\end{remark}

\subsection{Geck's cell modules} \label{geckcell1} Assume now that $W$ is a finite Coxeter group. A weight function is a function $L: W\rightarrow \mathbb{Z}$ such that $L(vw)=L(v)+L(w)$ whenever $\ell(vw)=\ell(v)+\ell(w)$, where $\ell$ denotes the length function for $W$. It is equivalent to a set of integral valued parameters $\{ L_{H,0}, L_{H,1} \}_{H\in \mathcal{A}}$ with $L(s_H) = L_{H,0} = - L_{H,1}$. This gives a cyclotomic Hecke algebra ${\mathcal H}_{q,L}$ over the Laurent polynomial ring $k=\C[q^{\pm 1}]$ defined by Hecke relations $(T_H - q^{L(s_H)})(T_H+q^{-L(s_H)}) = 0$ for all $H\in \mathcal{A}.$  The paper \cite{geck} defines a cellular algebra structure on ${\mathcal H}_{q,L}$ and hence on its specialisations, provided that the conjectures (P1)-(P15) of \cite[Conjectures 14.2]{lus} hold.
\subsection{} \label{geckcell} The general theory of cellular algebras produces a set of cell modules for ${\mathcal H}_{q,L}$, and each cell module carries a symmetric bilinear form, see for instance \cite[Example 4.4]{geck}. We can now identify these with the standard modules defined in \eqref{cell and simple} and their forms defined in Theorem \ref{forms}.
\begin{proposition} Let $\theta : {\bf k}\longrightarrow \C$ send $q_{H,0}$ to $exp(2\pi i \lambda L(s_H))$ and $q_{H,1}$ to $exp(-2\pi i \lambda L(s_H))$ for some $\lambda\in \C$. Assume that (P1)-(P15) hold so that $\mathcal{H}_{\theta}$ is a cellular algebra. Then for $E \in \irr{W}$ there is an isomorphism between the cell module $W_{\theta}(E)$ defined in \cite[Example 4.4]{geck} and the standard module $S_q(E)$ which preserves the symmetric bilinear forms.
\end{proposition}
\begin{proof}
Let $\Theta: {\bf k} \rightarrow k$ be defined by $q_{H,0} \mapsto q^{L(s_H)}$ and $q_{H,1} \mapsto q^{-L(s_H)}$ so that ${\mathcal H}_{q,L} = {\mathcal H}_{\Theta}$. Let ${\hat R}$ be the completion of ${\bf R}$ at the point $h_{H,0} = \lambda L(s_H), h_{H,1} = -\lambda L(s_H)$ and let $\C[[h]]$ be the completion of $\C[h]$ at $h-\lambda$. Then we have $\Psi: {\bf R}\rightarrow {\hat R} \rightarrow \C[[h]]$ be defined by $h_{H,0} \mapsto hL(s_H)$ and $h_{H,1} \mapsto -hL(s_H)$ and a commutative diagram $$\begin{CD} {\bf R} @>>> {\hat R} @>\Psi >> \C[[h]] \\ @. @AAA  @AAA \\ @. {\bf k} @> \Theta >> k,\end{CD}$$ where the right vertical map sends $q$ to $\exp(2\pi \sqrt{-1}h)$.

As explained in \cite[Example 4.4]{geck} the cell representations $W_{\theta}(E)$ are obtained by specialisation from cell modules $W_{q,L}(E)$ defined over $k$ (which themselves are constructed by pulling back representations from the asymptotic ring along Lusztig's homomorphism $\mathcal{H}_{q, L} \to J_{\Z}\otimes_{\Z} k$). Hence, for the first part of the proposition, it is enough to show that the cell representations on $k$ extended to $\C[[h]]$ are isomorphic to the representations ${\sf KZ}_{\Psi}(\Delta_{\Psi}(E))$ of Theorem \ref{forms}.  By \cite[\S 5, particularly Theorem 5.5 and the comments preceeding Remark 5.4]{Geck2} the cellular structure respects the decomposition of $\irr{W}$ into two-sided cells with respect to $L$, so the isomorphisms follow from the arguments of \cite[\S 6]{GGOR} and the characterisation of ${\sf KZ}_{\Psi}(\Delta_{\Psi}(E))$ we have given in the proof of Theorem \ref{forms} {\it provided} we know that $c$ is compatible with the ordering on two-sided cells. By \eqref{a+A}, we have that the ordering induced by $c$ is the same as that induced by $a+A$, and so it is enough to check that both $a$ and $A$ are compatible with the two-sided cells. This follows as explained in \cite[Remark 5.4]{Geck2} for $a$ and by \cite[Corollary 21.6]{lus} and \cite[Proposition 2.8]{ChlJa} for $A$.

Now we need to check that the forms agree (up to non-zero scalar). Let $\langle\,\, , \,\rangle'$ denote the bilinear form on $S_q(E)$ arising from the cellular structure, and continue with the notation from $\langle \,\, , \, \rangle$ for the form defined in Theorem \ref{forms}. Since the $c$-function is compatible with the ordering in the cell datum for ${\mathcal H}_{\theta}$, we see that the non-zero $S_q(E)/\text{rad}\langle\,\, ,\, \rangle'$ form a basic set with respect to the $c$-function, \cite[Proposition 3.6]{graleh}. Thus $D_q(E) \cong S_q(E)/\text{rad}\langle\,\, ,\, \rangle'$ and since $D_q(E)$ appears only once as a composition factor of $S_q(E)$, the radicals of $\langle \,\, , \, \rangle$ and $\langle \,\, , \, \rangle'$ are equal. It then follows that, up to scalar, the forms are the same.\end{proof}

\begin{remark}
The same proof works in type $G(\ell ,1 ,n)$ to show that the standard modules $S_q(E)$ agree with the Specht modules (intertwining the symmetric bilinear forms) defined by the cellular structure on the Ariki-Koike algebras in \cite{DJM}, provided that the parameters $h_{H,j}$ are chosen to belong to the ``asymptotic region" (see \cite[Proposition 6.4]{R} for the explicit description of this region). This proviso is necessary since it is only in this case that the ordering given by the $c$-function is compatible with the ordering by dominance of multi-partitions.
\end{remark}

\subsection{} \label{KZvanishes}Currently (P1)-(P15) are known to hold for all finite Coxeter groups except type $B_n$, where in general they are only known to hold in the asymptotic region, \cite{gecklusztIM} and \cite[Corollary 7.12]{geckrelative}. Nevertheless, we will give another argument in the next section that will imply the $B_n$ case of the following result.
\begin{corollary}
Let $W$ be a finite Coxeter group and let $\psi : {\bf R} \longrightarrow \C$ be defined by $\psi({\bf h}_{H,0}) = \lambda L(s_H)$ and $\psi({\bf h}_{H,1})=-\lambda L(s_H)$ for some weight function $L:W \rightarrow \mathbb{Z}$ and complex number $\lambda$. Then ${\sf KZ}_{\psi}(L_\psi(E)) \neq 0$ if and only if $E$ belongs to the corresponding canonical basic set.
\end{corollary}

The canonical basic sets are known explicitly, \cite{Gecksurv}.
\section{Type $G(\ell,1,n)$}

\subsection{} \label{paramchoiceforfamily}We are going to consider the case $W = G(\ell ,1,n)$ and in particular study the existence of canonical basic sets for the Hecke algebra $\mathcal{H}_{\theta}$ where $\theta$ is induced by $\psi({\bf h}_{H,j})  = \frac{s_j}{e} - \frac{j}{\ell}$ and $\psi({\bf h}_0) = \frac{1}{e}$, $\psi({\bf h}_1) = 0$ where $(s_0, \ldots , s_{\ell - 1}) \in \Z^{\ell}$ and $e \in \Z_{>0}$. 
In other words, we will study the Ariki-Koike algebra with relations 
\begin{equation}
(T_i-\zeta_e)(T_i+1) = 0, \qquad (T_0 - \zeta_e^{s_0})(T_0 - \zeta_e^{s_1})\cdots (T_0 - \zeta_e^{s_{\ell - 1}}) = 0.
\end{equation}
\subsection{} There are several different cyclotomic specialisations to the above Ariki-Koike algebra and they may have distinct $a$-functions attached to them. To deal with this generality we follow the combinatorial construction of the $a$-functions in \cite{GeJa} and show that they are all compatible with the highest weight structure on $\mathcal{O}_{\psi}$. 

To this end as well as the integer $e$ and the $\ell$-tuple $(s_0, \ldots , s_{\ell-1 })\in \mathbb{Z}^{\ell}$ we will need ${\bf u} = (u_0,\ldots,u_{\ell-1}) \in \mathbb{Q}^{\ell}$, a list of rational numbers such that $0<u_j-u_i<e$ whenever $i<j$. Set $t_j=s_j-u_j$, for all $0\leq j\leq \ell - 1$, and ${\bf t}=(t_0,\ldots,t_{\ell-1})$.

\subsection{} 
Recall that the irreducible representations of $G(\ell ,1,n)$ are labelled by the set of $\ell$-partitions of $n$, \cite[3.1]{Gr}. We will denote this by $\lambda \mapsto E^{\lambda}\in \irr{G(\ell , 1 ,n)}$. 

\subsection{} Given $\lambda=(\lambda^{(0)},\ldots,\lambda^{(\ell-1)})$ an $\ell$-composition of $n$, the set of nodes of $\lambda$ is the set
$$[\lambda]=\{ (a,b,c): 0 \leq c \leq \ell-1,\,\,a \geq 1,\,\,1 \leq b \leq \lambda_a^{(c)}\}.$$
Let $\gamma=(a(\gamma),b(\gamma),c(\gamma)) = (a,b,c)$  be a node of $\lambda$. We let
$$\mathrm{cont}(\gamma)=b-a,\,\,\,
\vartheta(\gamma)=\mathrm{cont}(\gamma)+s_{c}\,\,\,\text{and}\,\,\,
\eta(\gamma)=\mathrm{cont}(\gamma)+t_{c}.$$

Fix $z$ to be a positive integer greater than or equal to $n+1 - \min\{t_j\}$. Define for each $0\leq i \leq \ell-1$ the set $B_{r+t_i-1}(\lambda^{(i)})=\{ \lambda^{(i)}_t - t + t_i + z : 1 \leq t \leq r + [t_i]\}$ where $[t_i]$ denotes the integral part of $t_i$. Now let \begin{equation*} \label{betanumbers} \kappa_1 (\lambda) \geq \kappa_2(\lambda) \geq  \kappa_3(\lambda) \geq \cdots\end{equation*} be the elements of these sets, written in descending order. We will denote this list by $\kappa_{\bf t}(\lambda)$.

 \label{gloss} We define
  \begin{equation} \label{ndefn}n_{\bf t}(\lambda) = \sum_{i \geq 1} (i-1)\kappa_i(\lambda)\end{equation}
 and 
 \begin{equation} \label{adefn}a_{\bf t}(\lambda) = n_{\bf t}(\lambda)-n_{\bf t}(\emptyset).
 \end{equation} This depends on both ${\bf s}$ and on ${\bf u}$. If we choose $u_j = je/\ell$ for $0\leq j \leq \ell -1$ then $a_{\bf t}$ agrees with the definition of $a$-function given in \cite{jaca} and studied in the context of Uglov's work on canonical bases for higher level Fock spaces. On the other hand, in type $B$ ($\ell = 2$), another choice of ${\bf u}$ is presented in \cite[6.7]{GeJa} which produces the $a$-function arising from the Kazhdan-Lusztig theory for the Hecke algebras with unequal parameters as in \ref{geckcell1}. This definition is therefore captures all $a$-functions  for $G(\ell , 1,n)$ in the literature.

\subsection{} 
Generalising the dominance order for partitions, we will write $\kappa_{\bf t}(\lambda) \triangleleft \kappa_{\bf t}(\lambda')$ if $\kappa_{\bf t}(\lambda) \neq \kappa_{\bf t}(\lambda')$ and $ \sum_{i=1}^t \kappa_i(\lambda) \leq  \sum_{i=1}^t \kappa_i(\lambda')$ for all $t \geq 1$.  If
$\lambda$ and $\lambda'$ are $\ell$-compositions such that $\kappa_{\bf t}(\lambda) \triangleleft \kappa_{\bf t}(\lambda')$, then  $a_{\bf t}(\lambda)>  a_{\bf t}(\lambda')$.

\begin{lemma}\label{dominance}
Let $\mu$, $\mu'$ be $\ell$-compositions with $\kappa_{\bf t}(\mu) \trianglelefteq \kappa_{\bf t}(\mu')$.
Let $\lambda$ (respectively $\lambda'$) be an $\ell$-composition obtained from $\mu$ (respectively $\mu'$) by adding an extra node $\beta$ (respectively $\beta'$). If $\eta(\beta)< \eta(\beta')$, then
$\kappa_{\bf t}(\lambda) \triangleleft \kappa_{\bf t}(\lambda').$
\end{lemma}
\begin{proof}
Let
$ \kappa_{\bf t}(\mu) = \kappa_1 \geq \kappa_2 \geq \kappa_3 \geq \cdots $
 and
$\kappa_{\bf t}(\mu') = \kappa_1' \geq \kappa_2' \geq \kappa_3' \geq \cdots $.
By hypothesis
we have
$ \sum_{i=1}^t \kappa_i \leq  \sum_{i=1}^t \kappa_i'$, for all $t \geq 1$.

The nodes $\beta$ and $\beta'$ are added to the end of a row (which may be empty)
of the compositions $\mu$ and $\mu'$. This implies that there exist  $j$ and $j'$ such that
$$\kappa_{\bf t}(\lambda) = (\kappa_{\bf t}(\mu) \setminus \{\kappa_j\}) \cup \{\kappa_j+1\}
 \,\,\,\textrm{and}\,\,\,
 \kappa_{\bf t}(\lambda') = (\kappa_{\bf t}(\mu') \setminus \{\kappa_{j'}'\}) \cup \{\kappa_{j'}'+1\},$$
where $\kappa_j+1= \eta(\beta)+z$ and $\kappa_{j'}'+1=\eta(\beta')+z$. Since
$\eta(\beta)< \eta(\beta')$, we must have  $\kappa_j < \kappa_{j'}'$.
Setting $\kappa_0=\kappa_0' =\infty$, there exist $1\leq l \leq j$ 
and $1\leq l' \leq j'$
such that
$$\kappa_j \leq \kappa_l \leq \kappa_j+1 < \kappa_{l-1}
\,\,\,\textrm{and}\,\,\,
\kappa_{j'}'\leq \kappa_{l'}' \leq \kappa_{j'}'+1 < \kappa_{l'-1}'.$$
We then have
$ \kappa_{\bf t}(\lambda) = k_1 \geq k_2 \geq k_3 \geq \cdots 
$ and $
\kappa_{\bf t}(\lambda') = k_1' \geq k_2' \geq k_3' \geq \cdots, $
where
 $$
k_i := \left\{\begin{array}{ll}
\kappa_i & \textrm{ for } i<l \textrm{ or } i \geq j+1 ;\\
\kappa_j+1 & \textrm{ for } i=l;\\
\kappa_{i-1}& \textrm{ for } l<i <j+1,\\
\end{array}\right.
\,\,\,\textrm{and}\,\,\,
k_i' := \left\{\begin{array}{ll}
\kappa_i' & \textrm{ for } i<l' \textrm{ or } i \geq j'+1 ;\\
\kappa_{j'}'+1 & \textrm{ for } i=l';\\
\kappa_{i-1}'& \textrm{ for } l'<i <j'+1.\\
\end{array}\right.$$
One then shows
$\sum_{i=1}^t k_i \leq  \sum_{i=1}^t k_i'$
for all $t \geq 1$, and that there exists some $t$ such that the inequality is strict,
by distinguishing the six cases:
$$l' \leq l \leq j \leq j';\,\,
l' \leq l \leq j' \leq j;\,\,
l' \leq j'  <  l \leq j;\,\,
l \leq l' \leq j \leq j';\,\,
l \leq l' \leq j' \leq j;\,\,
l \leq j  <  l' \leq j'.$$

We conclude that 
$\kappa_{\bf t}(\lambda) \triangleleft \kappa_{\bf t}(\lambda').$
\end{proof}

A variation of the above lemma has been first used in the proof of \cite[Proposition 5.7.15]{GeJa}. It is needed for the proof of Proposition 5.6, which is a generalisation of the result of Geck and Jacon.

\subsection{}

 Let $\gamma$
and   $\gamma'$ be nodes of $\ell$-compositions. 
We write
$\gamma \prec \gamma'$
if we have
$\vartheta(\gamma)<\vartheta(\gamma')$ or if 
$\vartheta(\gamma)=\vartheta(\gamma') \,\textrm{ and }\,c(\gamma)>c(\gamma').$

\begin{proposition}
Let $\lambda$, $\lambda'$ be  $\ell$-compositions of $n$. Suppose that there exist orderings $\gamma_1,\gamma_2,\ldots,\gamma_n$ and $\gamma_1',\gamma_2',\ldots,\gamma_n'$ of the nodes of $\lambda$ and $\lambda'$ respectively such that,  for all $i=1,\ldots,n$,
$$\gamma_i \prec \gamma_i' \,\,\textrm{ or }\,\, \gamma_i = \gamma_i'.$$
Then either $\lambda=\lambda'$ or $a_{\bf t}(\lambda) > a_{\bf t}(\lambda')$.

\end{proposition} 

\begin{proof}
For all $1\leq i\leq n$, we have $\eta(\gamma_i)<\eta(\gamma_i')$,
unless $\gamma_i=\gamma_i'$. 
If there exist $1 \leq i \neq j \leq n$ such that
$\gamma_i \neq \gamma_i'$,
$\gamma_j \neq \gamma_j'$ and
 $\gamma_i'=\gamma_j$, 
 then we can exchange $\gamma_i'$ and $\gamma_j'$ in the ordering of the nodes of $\lambda'$ and get
$\eta(\gamma_i)<\eta(\gamma_i') \,\,\textrm{ and }\,\, \gamma_j=\gamma_j'.$
Therefore, we obtain orderings $\beta_1,\beta_2,\ldots,\beta_n$
and $\beta_1',\beta_2',\ldots,\beta_n'$ on the
 nodes of $\lambda$ and $\lambda'$ respectively such that for some $r \in \{0,1,\ldots,n\}$
\begin{itemize}
\item $\beta_i=\beta_i'$ \,for $i=1,\ldots,r$,\,
\item $\eta(\beta_i)<\eta(\beta_i')$ \,for $i=r+1,\ldots,n$\,
\item $\{\beta_{r+1},\ldots,\beta_n\} \cap \{\beta_{r+1}',\ldots,\beta_n'\}=\emptyset$.
\end{itemize}

Let $\mu$ be the $\ell$-composition defined by the nodes $\beta_1,\ldots,\beta_r$. 
If $r=n$, then $\lambda=\lambda'$. Otherwise, 
we have
$[\lambda]=[\mu] \cup \{\beta_{r+1},\ldots,\beta_{n}\}$ and $[\lambda']=[\mu] \cup \{\beta_{r+1}',\ldots,\beta_{n}'\},$
where
\begin{equation}\label{pairwise}
\eta(\beta_i)<\eta(\beta_i') \,\textrm{ for } r+1\leq i\leq n.
\end{equation}

Now, let $(b_1,b_2,\ldots,b_{n-r})$ be the nodes $\beta_{r+1},\ldots,\beta_{n}$ ordered with respect to increasing $\eta$-function and let $(b_1',b_2',\ldots,b_{n-r}')$ be the nodes $\beta_{r+1}',\ldots,\beta_{n}'$ ordered with respect to increasing $\eta$-function. 
We can then add the nodes $b_{1},\ldots,b_{n-r}$ 
(respectively $b_{1}',\ldots,b_{n-r}'$)
to $\mu$ in order to obtain $\lambda$ (respectively $\lambda'$) in the given order, \emph{i.e.,} we can always add the nodes $b_i$ and $b_{i}'$ at the same time, for all $i=1,\ldots,n-r$.
We will prove by induction that $\eta(b_i) < \eta (b_i')$ for all $i=1,\ldots,n-r$.

Take $1\leq t \leq n-r$ and assume that $\eta(b_i) < \eta (b_i')$ for all $i=1,\ldots,t-1$.
If $\eta(b_t) \geq \eta (b_t')$, then there exist only $t-1$ nodes in $\{\beta_{r+1},\ldots,\beta_{n}\}$
which have $\eta$-value less than $b_t'$. This contradicts Equation (\ref{pairwise}).
Hence, $\eta(b_t) < \eta (b_t')$.

We can now apply Lemma \ref{dominance} repeatedly to obtain that
$\kappa_{\bf t}(\lambda) \triangleleft \kappa_{\bf t}(\lambda'),$
whence $a_{\bf t}(\lambda)>a_{\bf t}(\lambda')$. 
\end{proof}

\subsection{} Now we can compare the ordering by $a_{\bf t}$ with the ordering on $\mathcal{O}_{\psi}$, where $\psi$ is defined in \ref{paramchoiceforfamily}. 

\begin{theorem}\label{DGGJ}
Let $\lambda$, $\lambda'$ be  $\ell$-partitions of $n$. If
$[\Delta_{\psi}(E^\lambda):L_{\psi}(E^{\lambda'})] \neq 0$ then 
 $\lambda=\lambda'$ or $a_{{\bf t}}(\lambda)>a_{\bf t}(\lambda')$.
\end{theorem} 
 
 \begin{proof}
Following  \cite[Proof of Theorem 4.1]{DG}, if
$[\Delta_{\psi}(E^\lambda):L_{\psi}(E^{\lambda'})] \neq 0$, then there exist orderings $\gamma_1,\gamma_2,\ldots,\gamma_n$ and $\gamma_1',\gamma_2',\ldots,\gamma_n'$ of the nodes of $\lambda$ and $\lambda'$ respectively, and non-negative integers $\mu_1,\mu_2,\ldots,\mu_n$  such that, for all $1\leq i\leq n$,
$$\mu_i \equiv c(\gamma_i)-c(\gamma_i')\,\,\mathrm{mod}\,\ell
\,\,\,\textrm{ and }\,\,\,
\mu_i= c(\gamma_i) - c(\gamma_i') + \frac{\ell }{e}(\vartheta(\gamma_i')-\vartheta(\gamma_i)).$$
If $\mu_i \geq \ell$, then $\vartheta(\gamma_i) < \vartheta(\gamma_i')$, whence
$\gamma_i \prec \gamma_i'$. Otherwise,
$$\mu_i=
\left\{\begin{array}{ll}
c(\gamma_i)-c(\gamma_i'), &\textrm{ if } c(\gamma_i) \geq c(\gamma_i')\\
\ell+c(\gamma_i)-c(\gamma_i'), &\textrm{ if } c(\gamma_i) < c(\gamma_i').
\end{array}\right.$$

Now if $c(\gamma_i) < c(\gamma_i')$, then $\vartheta(\gamma_i) < \vartheta(\gamma_i')$, whence
$\gamma_i \prec \gamma_i'$. If $c(\gamma_i) > c(\gamma_i')$, then $\vartheta(\gamma_i) =\vartheta(\gamma_i')$ and
$\gamma_i \prec \gamma_i'$. Finally, if $c(\gamma_i) = c(\gamma_i')$, then $\mathrm{cont}(\gamma_i) =\mathrm{cont}(\gamma_i')$ and $\gamma_i$ appears in $\lambda'^{(c(\gamma_i'))}$ or $\gamma_i'$ appears in $\lambda^{(c(\gamma_i))}$. In  either case, we can rearrange the nodes so that $\gamma_i=\gamma_i'$. 

We conclude that  there exist orderings $\gamma_1,\gamma_2,\ldots,\gamma_n$ and $\gamma_1',\gamma_2',\ldots,\gamma_n'$ of the nodes of $\lambda$ and $\lambda'$ respectively such that $\gamma_i \prec \gamma_i' \,\,\textrm{ or }\,\, \gamma_i = \gamma_i'$ for all $1\leq i \leq n$.
Proposition 5.4 thus completes the proof.
\end{proof}

\subsection{} Thanks to the above theorem and the comments in \ref{gloss} we see, by applying ${\sf KZ}_{\psi}$, that a canonical basic set exists for $\mathcal{H}_\theta$. This recovers \cite[Main Theorem]{jaca} but without using Ariki's theorem.

On the other hand, invoking Ariki's theorem gives us the following result, analogous to Corollary \ref{KZvanishes}. We refer the reader to \cite[Definition 3.2]{jaca} for the combinatorial definition of Uglov $\ell$-partitions.

\begin{corollary} \label{uglov} Let $W = G(\ell ,1,n)$. Let $e\in \mathbb{Z}_{>0}$, $(s_0, \ldots , s_{\ell -1})\in \mathbb{Z}^{\ell}$ and define $\psi: {\bf R}\rightarrow \C$ by $\psi({\bf h}_{H,j})  = \frac{s_j}{e} - \frac{j}{\ell}$ and $\psi({\bf h}_0) = \frac{1}{e}$, $\psi({\bf h}_1) = 0$. Then ${\sf KZ}_{\psi}(L_{\psi}(E^{\lambda}))\neq 0$ if and only if  $\lambda$ is an Uglov $\ell$-partition with respect to $(e\,;\,s_0, \ldots ,s_{\ell-1})$.
\end{corollary}

\subsection{} \label{guglov} Following the same reasoning as \cite[Theorem 3.1]{GeckJa} in the $\ell = 2$ case and using the Morita equivalences in \cite[Theorem 1.1]{DipMat}, we expect Corollary \ref{uglov} generalises to the case where  
$\psi({\bf h}_0) = \frac{k}{e}$ for $k \in  \mathbb{Z}_{>0}$. This, together with \cite[Proposition 2.5]{ChlJa}, would allow us to generalise the results of Genet and Jacon, \cite{GenJa}, to obtain canonical basic sets for the cyclotomic Hecke algebras associated with $G(\ell,p,n)$, in the cases where Clifford theory works: when $n>2$ or $n=2$ and $p$ is odd.

\bibliographystyle{plain}

\bibliography{canonical}

\end{document}